\documentclass[reqno]{amsart}
\usepackage{amsfonts}
\usepackage{amsmath}
\usepackage{amssymb}
\usepackage{mathrsfs}
\usepackage[colorlinks]{hyperref}
\usepackage{amsmath,tikz}
\usetikzlibrary{matrix}
\usepackage{algorithm}
\usepackage{algpseudocode}  
\usepackage{algorithmicx}
\usepackage{multirow}
\usepackage{graphicx,lscape}
\graphicspath{}
\DeclareGraphicsExtensions{.pdf,.png,.jpg}

\begin{document}
\title[\hfil Numerical identification of the time-dependent coefficient]
{Numerical Identification of a Time-Dependent Coefficient in a Time-Fractional Diffusion Equation with Integral Constraints}

\author[Arshyn Altybay]{Arshyn Altybay}
\address{
Arshyn Altybay: 
\endgraf
Institute of Mathematics and Mathematical Modeling,
\endgraf
 125 Pushkin str., 050010 Almaty, Kazakhstan
\endgraf
and
\endgraf
Department of Mathematics: Analysis, Logic and Discrete Mathematics, 
  \endgraf
Ghent University, Krijgslaan 281, Building S8, B 9000 Ghent, Belgium
  \endgraf
{\it E-mail address} {\rm arshyn.altybay@gmail.com, arshyn.altybay@ugent.be}
}

\thanks{}

\begin{abstract}
In this paper, we numerically address the inverse problem of identifying a time-dependent coefficient in the time-fractional diffusion equation. An a priori estimate is established to ensure uniqueness and stability of the solution. A fully implicit finite-difference scheme is proposed and rigorously analysed for stability and convergence. An efficient algorithm based on an integral formulation is implemented and verified through numerical experiments, demonstrating accuracy and robustness under noisy data.
\end{abstract}

\maketitle
\numberwithin{equation}{section} 
\newtheorem{theorem}{Theorem}[section]
\newtheorem{lemma}[theorem]{Lemma}
\newtheorem{corollary}[theorem]{Corollary}
\newtheorem{remark}[theorem]{Remark}
\newtheorem{definition}[theorem]
{Definition}
\newtheorem{proposition}[theorem]
{Proposition}
\allowdisplaybreaks


\section{Introduction}\label{1}
\noindent
We consider the inverse problem of identifying the pair of functions $\{p,u\}$ for the fractional diffusion equation
\begin{equation}\label{1.1}
\partial_t^{\alpha} u(x,t) - u_{xx}(x,t) + p(t)u(x,t) = f(x,t), \quad (x,t)\in Q_T,
\end{equation}
posed on the domain $Q_T:=\{(x,t):\, 0<x<l,\; 0\le t\le T\}$ with $T>0$, subject to the initial condition
\begin{equation}\label{1.2}
u(x,0)=\varphi(x), \quad x\in[0,l].
\end{equation}
and the following Dirichlet boundary conditions
\begin{equation}\label{1.3}
\begin{aligned}
u(0,t) = 0, \  t \in [0,T], \quad  u(l,t) = 0, \  t \in [0,T],
\end{aligned}
\end{equation}
and the integral overdetermination condition
\begin{equation}\label{1.4}
\int_0^lu(x,t)\omega(x) dx = g(t), \ t \in [0,T],
\end{equation}
and $f(x,t),\ \varphi(x),  \omega(x), \ g(t)$ are given functions.
Here, $u(x, t)$ represents the temperature, $p(t)$ describes the heat capacity, $ f(x, t)$ is a source function, $\partial^{\alpha}_{t}$ is the Caputo fractional derivative of order $\alpha \in (0,1)$ defined in \cite{A} as
\[
\partial^{\alpha}_{t}u(x,t)=\frac{1}{\Gamma(1-\alpha)}\int\limits_{0}^{t}\frac{u_s(x,s)}{(t-s)^{\alpha}}ds,
\]
where $\Gamma(\cdot)$ is the Gamma function. 

For a given sufficiently smooth $p$ and $f$, direct problem involves determining 
$u$ in $\overline{Q}_T$ such that 
$u(\cdot,t)\in C^2[0,l]$
and $u(x,\cdot)\in C[0,T]$ with $\partial^{\alpha}_{t}{u(x,\cdot)}\in C[0,T]$.

If the function $p$ is unknown, inverse problem is defined as problem of finding a pair of functions $\{p, u\}$ that satisfy \eqref{1.1}-\eqref{1.4}, with the additional constraints that $p\in C[0,T], \,\ g\in C^1[0,T],\,\ u(\cdot,t)\in C^2[0, l]$ and $u(x,\cdot)\in C[0,T]$ with $\partial^{\alpha}_{t}{u(x,\cdot)}\in C[0,T]$.

For the last decades, equation \eqref{1.1} has been widely recognized for its unique ability to model anomalously slow diffusion $-$ often termed subdiffusion. This framework has been applied across diverse fields, including thermal transport in fractal materials \cite{Nigmatullin86}, fluid movement in porous underground formations \cite{Adams92}, and the passage of proteins through cellular membranes \cite{Kubica20}. For an in-depth discussion of the physical underpinnings and a broad survey of applications in both physics and engineering, see the survey article \cite{Metzler20}.

Since fractional derivatives are non-local with singular kernels, developing robust and accurate schemes plays a key role in fractional-time diffusion solvers. Many methods have been analysed: the $L1$ scheme \cite{Langlands05, Sun06, Lin07, Stynes22, Shen24, Zhou24}, the $L2$ scheme \cite{Alikhanov15, Gao14, Lv16, Alikhanov21, Xing18}, and approaches on nonuniform time meshes \cite{Kopteva19, Wang21, Stynes17, Chen19, Liao21}. In particular, Kopteva \cite{Kopteva19}, Wang et al. \cite{Wang21}, and Stynes et al. \cite{Stynes17} studied $L1$ on nonuniform meshes using different analyses, while Chen \& Stynes \cite{Chen19} and Liao et al. \cite{Liao21} extended the $L2-1\sigma$ method from uniform to nonuniform meshes.

The integral overdetermination condition given in \eqref{1.4} defines a key category of inverse problems where observed data consists of spatially averaged quantities rather than localized measurements. Such problems frequently arise in real-world applications $-$ for instance, when sensors measure cumulative values like total contaminant concentration in environmental studies or aggregated thermal energy in heat transfer systems.
The weighting function $\omega(x)$ reflects the spatial response characteristics of the measurement instrument, whereas $g(t)$ supplies the temporal evolution of the observed integral data, which is crucial for determining unknown parameters. This formulation was thoroughly examined by Prilepko and collaborators in their work \cite{Prilepko2000}.
Further studies have expanded on this concept, with significant contributions found in recent works such as \cite{Grim20}, \cite{VanBock22}, \cite{Hendy22}, \cite{Karel22}, \cite{Durd24}, \cite{Durd25}, and \cite{DurdRah25}, demonstrating its continued relevance in contemporary research.

Many authors \cite{Fujishiro16}, \cite{Jin21}, \cite{Ozbilge22}, \cite{Durdiev23}, \cite{Jumaev23}, \cite{Jin24}, \cite{SC24}, \cite{Bekbolat19}, \cite{Bekbolat24} have studied the inverse problem of recovering the time-dependent reaction coefficient in equation \eqref{1.1} under various boundary conditions and additional data.

In more detail, Fujishiro \& Kian (2016) \cite{Fujishiro16} proved uniqueness and Hölder-type stability when the solution is observed at a single interior point (or, equivalently, its Neumann trace on boundary), showing this minimal extra datum suffices to recover $p(t)$. 
Jin \& Zhou (2021) \cite{Jin21} recovered both the fractional order and a spatial potential by prescribing full Cauchy data (Dirichlet + flux) at one boundary endpoint for all $t$, while keeping the initial condition and source unknown. 
Özbilge et al. (2022) \cite{Ozbilge22} dealt with non-local integral boundary conditions and used an overspecified Dirichlet trace on the boundary as the additional measurement driving their finite-difference iteration. 
Durdiev \& Durdiev (2023) \cite{Durdiev23} imposed the usual Dirichlet boundary but appended an extra Neumann condition at $x=0$, converting the inverse task to a Volterra equation that yields global uniqueness. 
Durdiev \& Jumaev (2023) \cite{Jumaev23} extended this to bounded multi-D domains, again relying on an overdetermined Neumann boundary trace to establish existence, uniqueness and stability. 
Jin, Shin \& Zhou (2024) \cite{Jin24} assumed the spatial integral of $u(x,t)$ over the whole domain is known for every $t$, derived conditional Lipschitz stability, and built a fast fixed-point recovery. 
Cen, Shin \& Zhou (2024) \cite{SC24} showed that a single-point boundary flux measurement already guarantees Lipschitz stability and designed a graded-mesh FEM algorithm for practical reconstruction. 

In \cite{RF}, the authors examine the semilinear variant of (\ref{1.1}), addressing the inverse problem of reconstructing the time-dependent reaction coefficient in a Caputo-fractional reaction–diffusion equation ($0<\alpha<1$),subject to nonlocal boundary conditions and integral-redefinition constraints.

Despite their advances, these studies leave notable numerical gaps. Most reconstructions were tested with minimal exploration of mesh adaptivity or high-dimensional efficiency. Error bounds are often asymptotic or conditional, lacking rigorous estimates and noise-propagation analysis. Moreover, convergence proofs rarely cover fully discrete schemes; stability is inferred from continuous theory rather than demonstrated for practical algorithms. Hence, a systematic, global error analysis - coupled with robust, noise-aware discretisations - remains an open task.

The task of recovering the time‐dependent coefficient $p(t)$ in the classical parabolic equation has been addressed by many authors under a variety of overdetermination conditions and boundary setups (e.g.\ \cite{Cannon92, Dehghan01, Ivanchov01, Fatullayev08, Kerimov12, Kanca13, Hazanee13, Hazanee14, Durdiev22}). Likewise, several numerical algorithms have been developed to reconstruct $p(t)$ from integral-type observations (see, for instance, \cite{Cannon92, Kanca13, Hazanee13, Hazanee14, Durdiev22}). 

This paper addresses a significant deficiency in the existing literature by providing robust numerical solutions to problems \eqref{1.1}-\eqref{1.4}. First, we impose a nonhomogeneous integral overdetermination condition as an auxiliary constraint, greatly broadening the framework’s applicability to practical scenarios. Second, we design an unconditionally stable, fully implicit finite‐difference solver and rigorously establish its stability and convergence properties. Third, we develop an efficient, noise-robust computational algorithm for simultaneously identifying a pair of functions $\{p, u\}$.

The remainder of this paper is organised as follows.  In Section \ref{2}, we present key inequalities and foundational lemmas that are repeatedly employed in the analysis throughout the paper. In Section \ref{3}, we derive an a priori estimate that ensures the uniqueness of the direct solution and its continuous dependence on the initial data for problem \eqref{1.1}-\eqref{1.3}.  Section \ref{4} introduces a fully implicit finite-difference scheme and rigorously analyses its stability and convergence. In Section \ref{5}, we develop a computational procedure based on integral formulations to tackle the inverse problem. Finally, Section \ref{6} presents numerical experiments and discusses their results.

\section{Preliminaries}\label{2}
In this section, we recall some fundamental inequalities and lemmas used throughout the paper.

\begin{lemma} \label{lemma1} (Alikhanov \cite{Alikhanov10})
For an arbitrary absolutely continuous function $u(t)$ defined on the interval $[0,T]$, the following inequality holds:
\begin{equation}
u(t) \partial^{\alpha}_{t} u(t) \geq \frac{1}{2} \partial^{\alpha}_{t} u^2(t), \quad 0<\alpha<1.
\end{equation}
\end{lemma}

\begin{lemma}\label{lemma2}(Poincaré’s inequality {\cite[Prop.~8.13, p.~274]{Brezis2010}})
\[
\int_0^l u^2(x)\,dx \le \frac{l^2}{\pi^2} \int_0^l u_x^2(x)\,dx, \quad \forall u \in H^1_0(0,l).
\]
\end{lemma}

\begin{lemma}\label{lemma3}(Cauchy–Schwarz inequality  \cite{Evans2010})
Let \(f\) and \(g\) be measurable functions on a domain \(\Omega\) such that \(f,g\in L^2(\Omega)\).  Then
\[
\left|\int_\Omega f(x)\,g(x)\,\mathrm{d}x\right|
\;\le\;
\Bigl(\int_\Omega |f(x)|^2\,\mathrm{d}x\Bigr)^{1/2}
\;\Bigl(\int_\Omega |g(x)|^2\,\mathrm{d}x\Bigr)^{1/2}.
\]
\end{lemma}

\begin{lemma}\label{lemma4}($\varepsilon$ Cauchy inequality \cite{Evans2010})
For any real numbers \(a,b\) and any \(\varepsilon>0\), one has
\[
a\,b\;\le\;\varepsilon\,a^2 \;+\;\frac{1}{4\varepsilon}\,b^2.
\]
\end{lemma}

\begin{definition}\label{definition1}
Let \( \psi(x,t) \in L^2(0,l) \) for each \( t \in [0,T] \), and let \( \alpha \in (0,1) \). We define the Riemann–Liouville fractional integral of the squared spatial \( L^2 \)-norm of \( \psi \) as
\[
\mathcal{D}_t^{-\alpha} \| \psi(\cdot,t) \|_{L^2(0,l)}^2 := \frac{1}{\Gamma(\alpha)} \int_0^t \frac{\| \psi(\cdot,s) \|_{L^2(0,l)}^2}{(t - s)^{1 - \alpha}} \, ds,
\]
where
\[
\| \psi(\cdot,t) \|_{L^2(0,l)}^2 = \int_0^l |\psi(x,t)|^2\,dx.
\]
\end{definition}

\section{A priori estimate}\label{3}
This section establishes an appropriate a priori estimate to ensure the uniqueness of the direct solution and the continuous dependence on the initial data.
\begin{theorem}\label{Theorem 1} If \(p(t) \geq 0\),\,\,\( p(t) \in C([0,T]), \,\, f(x,t)\in C(\overline{Q}_T)\) everywhere in \(Q_T\) then the solution \(u(x,t)\) of the problem \eqref{1.1}-\eqref{1.3} satisfies the a priori estimate:
 
\begin{equation} \label{theorem1}
\| u(x,t)\|_{L^2}^2 \leq \| \varphi(x) \|_{L^2}^2  + c\mathcal{D}_t^{-\alpha} \| f(x,t)\|_{L^2}^2,
\end{equation}
where $ \| u(x,t) \|_{L^2}^2 = \int_0^l u^2(x,t)dx$, $c=\frac{l^2}{2\pi^2}$.
\end{theorem}
\begin{proof}\label{proof 1} 
We multiply equation \eqref{1.1} by \( u(x,t) \) and integrate over \( x \in (0,l) \):
\begin{equation} \label{Theo1_eq1}
\int_0^l u(x,t) \partial^{\alpha}_{t}{u}(x,t) \,dx  - \int_0^l u(x,t) u_{xx}(x,t) \,dx + \int_0^l p(t) u^2(x,t) \,dx = \int_0^l f(x,t) u(x,t) \,dx.
\end{equation}
Then we transform the terms of identity \eqref{Theo1_eq1}; by Lemma~\ref{lemma1}, we obtain
\[ 
\int_0^l u(x,t) \partial^{\alpha}_{t} u(x,t) \,dx \geq  \int_0^l \frac{1}{2} \partial^{\alpha}_{t}u^2(x,t) \,dx = \frac{1}{2} \partial^{\alpha}_{t} \| u(x,t) \|_{L^2}^2.
\]
Using integration by parts:
\[
 - \int_0^l u(x,t) u_{xx}(x,t) \,dx = -\left[ u(x,t) u_x(x,t) \right]_0^l + \int_0^l u_x^2(x,t) \,dx.
\]
From the Dirichlet boundary conditions \eqref{1.3},
\[
u(0,t) = 0, \quad u(l,t) = 0,
\]
thus,\[
-\left[ u u_x \right]_0^l = -\left[u(l,t) u_x(l,t) - u(0,t) u_x(0,t) \right] =  0,
\]

\[
- \int_0^l u(x,t) u_{xx}(x,t) \,dx =  \int_0^l u_x^2(x,t) = \| u_x(x,t) \|_{L^2}^2.
\]

To bound the right-hand side, we use the Cauchy-Schwarz inequality (Lemma \ref{lemma3})
\[
\left| \int_0^l f(x,t) u(x,t) \,dx \right| \leq  \| u(x,t) \|_{L^2} \| f(x,t) \|_{L^2}.
\]
Using $\varepsilon$ Cauchy inequality (Lemma \ref{lemma4})
\[
\left| \int_0^l f(x,t) u(x,t) \,dx \right| \leq  \frac{1}{4\varepsilon}\| f(x,t) \|_{L^2}^2 + \varepsilon \| u(x,t) \|_{L^2}^2.
\]

Substituting all the terms into the identity \eqref{Theo1_eq1}, we obtain:
\begin{equation} \label{Theo1_eq2}
\frac{1}{2}\partial^{\alpha}_{t} \| u(x,t) \|_{L^2}^2 + \| u_x(x,t) \|_{L^2}^2 \leq  \frac{1}{4\varepsilon}\| f(x,t) \|_{L^2}^2 + \varepsilon\| u(x,t)\|_{L^2} ^2.
\end{equation}

Using the Poincaré's inequality (Lemma \ref{lemma2}) and set $\varepsilon = \frac{\pi^2}{l^2}$, we obtain:
\begin{equation} \label{Theo1_eq2}
\partial^{\alpha}_{t} \| u(x,t) \|_{L^2}^2  \leq  \frac{l^2}{2\pi^2}\| f(x,t) \|_{L^2}^2.
\end{equation}

Applying the fractional integral operator \( \mathcal{D}_t^{-\alpha} \) (Definition \ref{definition1}) both sides of inequality \eqref{Theo1_eq2}, we obtain
\begin{equation}\label{Theo1_eq2_frac_integrated}
\mathcal{D}_t^{-\alpha} ( \frac{1}{2} \partial^{\alpha}_{t} \| u(x,t) \|_{L^2}^2 )
\leq \mathcal{D}_t^{-\alpha}  (\frac{l^2}{2\pi^2}\| f(x,t) \|_{L^2}^2).
\end{equation}

Using the property
\[
\mathcal{D}_t^{-\alpha}  \partial^{\alpha}_{t} \| u(x,t)\|_{L^2}^2 = \| u(x,t)\|_{L^2}^2 - \| u(x,0)\|_{L^2}^2,
\]
inequality \eqref{Theo1_eq2_frac_integrated} simplifies to
\begin{equation}\label{Theo1_eq2_final}
\| u(x,t)\|_{L^2}^2 \leq \| \varphi(x) \|_{L^2}^2  + \frac{l^2}{2\pi^2}\mathcal{D}_t^{-\alpha}  \| f(x,t)\|_{L^2}^2.
\end{equation}
\end{proof}
The a priori estimate \eqref{Theo1_eq2_final} ensures both the uniqueness and continuous dependence
of the direct solution to problems \eqref{1.1}-\eqref{1.3} on the initial data.

\subsection{Existence and Uniqueness of the Inverse Problem}

In this subsection, we present the theoretical results on the existence and uniqueness of solutions to the inverse problem \eqref{1.1}–\eqref{1.4}, as established in \cite{DurdRah25}. These results hold under the following set of assumptions:

\begin{itemize}
\item[(A1)] \( \varphi \in H^2(0, l) \cap H_0^1(0, l), \quad f \in C([0, T]) \);
\item[(A2)] \( g(0) + \beta g(T) = (\omega, \varphi), \quad \beta \geq 0 \);
\item[(A3)] \( \partial_t^\alpha g \in C[0, T] \) and satisfies the positivity condition:
\[
|g(t)| \geq \frac{1}{g_0} > 0, \quad \text{for all } t \in [0, T],
\]
where \( g_0 \) is a positive constant;
\item[(A4)] \( \omega(x) \in L^2(0, l) \).
\end{itemize}

\begin{theorem}[Existence, \cite{DurdRah25}]
Assume conditions $(A1)-(A4)$ are satisfied. Then, for any \( T > 0 \), there exists at least one solution \((u, p) \in C([0, T]; H^2(0, l)) \times C([0, T])\) to the inverse problem \eqref{1.1}-\eqref{1.4}.
\end{theorem}

\begin{theorem}[Uniqueness, \cite{DurdRah25}]
Let \( T > 0 \) and suppose assumptions $(A1)-(A4)$ hold. If the inverse problem \eqref{1.1}-\eqref{1.4} admits two solutions \((u_i, p_i) \in C([0, T]; H^2(0, l)) \times C([0, T])\) for \( i = 1, 2 \), then these solutions coincide, i.e.,
\[
(u_1(t), p_1(t)) = (u_2(t), p_2(t)) \quad \text{for all } t \in [0, T].
\]
\end{theorem}
In the remainder of this work, we focus on the numerical solution of the inverse problem \eqref{1.1}--\eqref{1.4}. Building upon the established existence and uniqueness results, we develop and implement stable numerical schemes to approximate both the state variable \( u(x,t) \) and the unknown time-dependent coefficient \( p(t) \). Our goal is to construct stable and accurate methods that efficiently identify the inverse data from available measurements.

\section{Numerical Methods for the direct problem}\label{4}
In this section, we present a numerical solution of the time-dependent fractional diffusion
equation \eqref{1.1} with initial \eqref{1.2} and boundary \eqref{1.3} using the finite difference method.
We divide the spatial domain $[0,l]$ into $N+1$ grid points with spacing $h=l/N$, and the time domain $[0, T]$ into $M+1$ grid points using a \emph{graded} mesh
\[
t_k \,=\, T\Big(\frac{k}{M}\Big)^{r}, \qquad k=0,1,\dots,M,\qquad r\ge 1,
\]
with variable steps $\tau_k:=t_k-t_{k-1}$. Let the grid points be denoted as $x_i=ih$, $i=0,1,\dots, N$, and let $u_i^k$ be the numerical
approximation to $u(x_i, t_k)$. The Caputo fractional derivative is approximated by the (nonuniform) L1 formula:
\begin{equation}\label{3.1}
\frac{\partial^\alpha u(x_i,t_k)}{\partial t^\alpha}
\;\approx\;
\frac{1}{\Gamma(2-\alpha)}\sum_{j=1}^{k} d_{k,j}\,\big(u_i^{j}-u_i^{j-1}\big),
\end{equation}
where, for $1\le j\le k$,
\[
d_{k,j}
\;:=\;
\frac{(t_k-t_{j-1})^{\,1-\alpha}-(t_k-t_j)^{\,1-\alpha}}{\tau_j}.
\]
(Here, $0<\alpha<1$.  For fixed $k$, the weights satisfy $d_{k,1} \le d_{k,2}\le...\le d_{k,k}$ (in particular $d_{k,j+1} - d_{k,j} \ge 0$ and $d_{k,k} = \tau_k^{-\alpha}$). Using formula \eqref{3.1} together with a central difference scheme of order $O(h^2)$ to construct an implicit difference scheme for \eqref{1.1}, we obtain the following.

Then problem \eqref{1.1} can be rewritten in the following form for $i=1,\dots,N-1$ and $k=1,\dots,M$:
\begin{equation}\label{3.2}
\frac{1}{\Gamma(2-\alpha)} \sum_{j=1}^{k} d_{k,j}\,\big(u_i^{j}-u_i^{j-1}\big)
\;-\; \frac{u_{i+1}^{k} - 2u_i^{k} + u_{i-1}^{k}}{h^2}
\;+\; p^k u_i^k \;=\; f_i^k.
\end{equation}

Then, let us rearrange:
\[
\frac{1}{\Gamma(2-\alpha)}\!\left[\, d_{k,k}u_i^{k}
+ \sum_{j=1}^{k-1} d_{k,j}\,u_i^{j}
- \Big( d_{k,1}u_i^{0} + \sum_{j=1}^{k-1} d_{k,j+1}\,u_i^{j}\Big)\right]
-
\frac{u_{i+1}^{k} - 2u_i^{k} + u_{i-1}^{k}}{h^2}
+ p^k u_i^k = f_i^k,
\]
\[
\frac{1}{\Gamma(2-\alpha)}\!\left[\, d_{k,k}u_i^{k} - d_{k,1}u_i^{0}
- \sum_{j=1}^{k-1}\big(d_{k,j+1}-d_{k,j}\big)\,u_i^{j}\right]
-
\frac{u_{i+1}^{k} - 2u_i^{k} + u_{i-1}^{k}}{h^2}
+ p^k u_i^k = f_i^k,
\]
and, writing the history terms with $u_i^{k-j}$ as in the uniform case,
\[
\frac{1}{\Gamma(2-\alpha)}\!\left[\, d_{k,k}u_i^{k} - d_{k,1}u_i^{0}
- \sum_{j=1}^{k-1}\big(d_{k,k-j+1}-d_{k,k-j}\big)\,u_i^{\,k-j}\right]
-
\frac{u_{i+1}^{k} - 2u_i^{k} + u_{i-1}^{k}}{h^2}
+ p^k u_i^k = f_i^k.
\]
Thus, we obtain the tridiagonal form
\begin{equation}\label{3.3}
-\frac{1}{h^2} u_{i+1}^{k}
+ \Big(\frac{d_{k,k}}{\Gamma(2-\alpha)} + \frac{2}{h^2} + p^k \Big) u_{i}^{k}
-\frac{1}{h^2} u_{i-1}^{k}
=
\frac{d_{k,1}}{\Gamma(2-\alpha)}\,u_i^0
+ \frac{1}{\Gamma(2-\alpha)} \sum_{j=1}^{k-1}\!\big(d_{k,k-j}-d_{k,k-j+1}\big)\,u_i^{\,k-j}
+ f_i^k .
\end{equation}

The discretisation of the conditions \eqref{1.2} and \eqref{1.3} give
\[
u_i^0 = \varphi(x_i), \quad i = 0,1,\dots,N,
\]
\begin{equation}\label{3.4}
u_0^k = 0, \quad u_N^k = 0, \quad \text{for all } k = 0,1,\dots,M.
\end{equation}
Combining all the difference equations for points $x_i$ (for $i=1,\dots,N-1$), we can form the system of $(N-1)$ linear algebraic equations in matrix form

\[
u^k = \begin{pmatrix} u_1^k \\ u_2^k \\ \vdots \\ u_{N-1}^k \end{pmatrix},
\qquad
A^k =
\begin{pmatrix}
c_1^k & -\tfrac1{h^2} &        &        &   \\
-\tfrac1{h^2} & c_2^k & -\tfrac1{h^2} &        &   \\
      & \ddots & \ddots & \ddots &   \\
      &        & -\tfrac1{h^2} & c_{N-2}^k & -\tfrac1{h^2} \\
      &        &        & -\tfrac1{h^2} & c_{N-1}^k
\end{pmatrix},
\quad
b^k =
\begin{pmatrix}
b_1^k \\ b_2^k \\ \vdots \\ b_{N-2}^k \\ b_{N-1}^k
\end{pmatrix},
\]

where
\begin{equation}\label{3.5}
c_i^k = \frac{d_{k,k}}{\Gamma(2-\alpha)} + \frac{2}{h^2} + p^k,
\qquad
b_i^k = \frac{d_{k,1}}{\Gamma(2-\alpha)}\,u_i^0
+ \frac{1}{\Gamma(2-\alpha)}
\sum_{j=1}^{k-1}\!\big(d_{k,k-j}-d_{k,k-j+1}\big)\,u_i^{\,k-j}
+ f_i^k,
\end{equation}
so that the scheme is compactly
\begin{equation}\label{tri}
A^k\,u^k \;=\; b^k.
\end{equation}
Here, the scheme \eqref{tri} is a tridiagonal system of equations, so we can solve it using the Thomas method.

\subsection{Stability and convergence of the implicit difference schemes}

To justify the proposed algorithm, we will derive estimates of the stability of the scheme \eqref{3.3}
concerning the initial data and the right-hand side.

\begin{theorem}[Stability]\label{thm:stability}
\textit{Let $0<\alpha<1$ and let $\{u_i^k\}_{0\le i\le N,\;0\le k\le M}$ be the solution of the finite
difference scheme \eqref{3.3} on the graded time grid $t_k=T(k/M)^r$ with $r\ge1$, with
Dirichlet conditions $u_0^k=u_N^k=0$ and initial data $u_i^0=\varphi(x_i)$. Assume the coefficient
$p^{k}\ge 0$ for all $k$.}
\textit{Then the scheme is unconditionally stable in the maximum norm ($L^\infty$): for any spatial
mesh size $h>0$ and graded time grid, the solution satisfies}
\begin{equation}\label{eq:Linfty-bound-graded}
\|u^k\|_{\infty}\;\le\;
\|\varphi\|_{\infty}
+ C_\alpha\,\max_{1\le m\le k}\|f^{\,m}\|_{\infty},
\qquad 1\le k\le M,
\end{equation}
\textit{where $\|u^k\|_\infty = \max_{0 \leq i \leq N} |u_i^k|$ and one can choose}
\[
C_\alpha \;=\; \frac{\Gamma(2-\alpha)}{\underline{d}_1}
\quad\text{with}\quad
\underline{d}_1:=\min_{1\le m\le M} d_{m,1}
\;\;\;\text{and}\;\;\;
d_{k,1}:=\frac{(t_k-t_0)^{1-\alpha}-(t_k-t_1)^{1-\alpha}}{\tau_1}.
\]
\textit{In particular, by the mean value theorem $d_{k,1}=(1-\alpha)\xi^{-\,\alpha}$ for some
$\xi\in(t_k-\tau_1,t_k)\subset(0,T]$, hence $\underline{d}_1\ge (1-\alpha)T^{-\alpha}$ and}
\begin{equation}\label{eq:Linfty-bound-graded-simple}
\|u^k\|_{\infty}\;\le\;
\|\varphi\|_{\infty}
+ T^{\alpha}\Gamma(1-\alpha)\,
\max_{1\le m\le k}\|f^{\,m}\|_{\infty},
\qquad 1\le k\le M,
\end{equation}
\textit{with a constant independent of $h$, $M$, and the grading parameter $r\ge1$.}
\end{theorem}

\begin{proof}
\medskip\noindent
For a fixed time level $k$ introduce the interior vector
$u^{k}=(u_1^{k},\dots,u_{N-1}^{k})^{\top}$ and write the graded scheme \eqref{3.3} in matrix form
\begin{equation}\label{eq:matrix-graded}
   A^{k}u^{k}=b^{k}.
\end{equation}
The coefficient matrix $A^{k}\in\mathbb{R}^{(N-1)\times(N-1)}$ is tridiagonal with
\[
  a^{k}_{ii}= \frac{d_{k,k}}{\Gamma(2-\alpha)}
                +\frac{2}{h^{2}} + p(t_k),\qquad
  a^{k}_{i,i\pm1}=-\frac{1}{h^{2}},\qquad
  d_{k,k}=\tau_k^{-\alpha},\ \ \tau_k:=t_k-t_{k-1}.
\]
Because $p(t_k)\ge 0$ we have 
$a^{k}_{ii} > |a^{k}_{i,i-1}|+|a^{k}_{i,i+1}|$,  
so $A^{k}$ is a \emph{strictly diagonally dominant $M$–matrix}; see~\cite[Chap.~6]{Varga09}.  
Consequently,
\begin{equation}\label{eq:M-properties-graded}
  A^{k}\ \text{is nonsingular and } (A^{k})^{-1}\ge 0\ \text{(entry-wise)}.
\end{equation}

Taking maximum norms in \eqref{eq:matrix-graded} and using \eqref{eq:M-properties-graded} yields
\begin{equation}\label{eq:linfty-step1-graded}
   \|u^{k}\|_{\infty}\le
   \|(A^{k})^{-1}\|_{\infty}\,\|b^{k}\|_{\infty}.
\end{equation}
Since $A^{k}$ is strictly diagonally dominant with diagonal dominance
$a^{k}_{ii}-\sum_{j\neq i}|a^{k}_{ij}|=d_{k,k}/\Gamma(2-\alpha)+p(t_k)$,
we have
\begin{equation}\label{eq:invbound-graded}
   \|(A^{k})^{-1}\|_{\infty}
   \le \frac{\Gamma(2-\alpha)}{d_{k,k}}
   = \Gamma(2-\alpha)\,\tau_k^{\alpha};
   \quad \text{see \cite[Thm.~6.4]{Varga09}.}
\end{equation}

\medskip\noindent
\emph{Estimate of the right–hand side.}
The vector $b^k$ is given by
\[
 b^{k}=
 \frac{1}{\Gamma(2-\alpha)}
 \Bigl[d_{k,1}\,u^{0}
       +\sum_{j=1}^{k-1}\big(d_{k,k-j+1}-d_{k,k-j}\big)\,u^{\,k-j}\Bigr]
  +f^{\,k},
\]
where $d_{k,j}$ is defined as in \eqref{3.1} and the differences
$\delta_{k,j}:=d_{k,k-j+1} - d_{k,k-j}\ge 0$ because $d_{k,j}$ is increasing in $j$.
Hence
\begin{equation}\label{eq:b-bound-graded}
  \|b^{k}\|_{\infty}\le
  \frac{d_{k,1}}{\Gamma(2-\alpha)}\,\|\varphi\|_{\infty}
 +\frac{d_{k,k}-d_{k,1}}{\Gamma(2-\alpha)}
     \max_{1\le m\le k-1}\|u^{m}\|_{\infty}
 +\|f^{\,k}\|_{\infty},
\end{equation}
since $\sum_{j=1}^{k-1}\delta_{k,j}=d_{k,k}-d_{k,1}$.

\medskip\noindent
\emph{Induction on $k$.}
Combining \eqref{eq:linfty-step1-graded}–\eqref{eq:b-bound-graded} gives
\[
  \|u^{k}\|_{\infty}
  \le \underbrace{\tau_k^{\alpha} d_{k,1}}_{=: \Theta_k}\,\|\varphi\|_{\infty}
       +\bigl(1-\Theta_k\bigr)
         \max_{1\le m\le k-1}\|u^{m}\|_{\infty}
       +\Gamma(2-\alpha)\tau_k^{\alpha}\,\|f^{\,k}\|_{\infty},
\]
because $\tau_k^\alpha(d_{k,k}-d_{k,1})=1-\tau_k^\alpha d_{k,1}$. Note that $0<\Theta_k\le1$
(since $d_{k,1}\le d_{k,k}=\tau_k^{-\alpha}$). For $k=1$, $\Theta_1=1$ and
$\|u^{1}\|_\infty\le \|\varphi\|_\infty+\Gamma(2-\alpha)\tau_1^\alpha\|f^{\,1}\|_\infty$.
Assume now that
$\|u^{m}\|_{\infty}\le \|\varphi\|_{\infty}
       +C_\alpha\max_{1\le\ell\le m}\|f^{\ell}\|_{\infty}$
holds for all $m\le k-1$. Then
\[
  \|u^{k}\|_{\infty}
  \le \|\varphi\|_{\infty}
     +\Bigl[(1-\Theta_k)C_\alpha+\Gamma(2-\alpha)\tau_k^{\alpha}\Bigr]
       \max_{1\le \ell\le k}\|f^{\ell}\|_{\infty}.
\]
Choosing $C_\alpha$ so that
$C_\alpha \ge \Gamma(2-\alpha)\tau_k^\alpha/\Theta_k$ for all $k$ yields
\[
  \|u^{k}\|_{\infty}\le
     \|\varphi\|_{\infty}
     +C_\alpha\max_{1\le \ell\le k}\|f^{\ell}\|_{\infty}.
\]
Since $\Theta_k=\tau_k^\alpha d_{k,1}$, one may take
\(
C_\alpha=\sup_k \Gamma(2-\alpha)/d_{k,1}=\Gamma(2-\alpha)/\underline{d}_1.
\)
Finally, by the mean value theorem $d_{k,1}=(1-\alpha)\xi^{-\alpha}$ with
$\xi\in(0,T]$, hence $\underline{d}_1\ge(1-\alpha)T^{-\alpha}$ and the simplified bound
\eqref{eq:Linfty-bound-graded-simple} follows. This completes the proof.
\end{proof}

\paragraph{Assumption.}
Let $0<\alpha<1$, $p\in C([0,T])$ with $p\ge0$, $\varphi\in H^2(0,\ell)\cap H_0^1(0,l)$,
and $f\in C(\overline{Q}_T)$.
Assume the exact solution $u$ of \eqref{1.1}--\eqref{1.3}
satisfies, for all $(x,t)\in[0,l]\times(0,T]$,
\[
\bigl|\frac{\partial^{\,k}u}{\partial x^{k}}(x,t)\bigr|\le C \quad \text{for } k=0,1,2,3,4,
\qquad
\bigl|\frac{\partial^{\,\ell}u}{\partial t^{\ell}}(x,t)\bigr|\le C\bigl(1+t^{\alpha-\ell}\bigr)
\quad  \text{for } \ell=0,1,2.
\]
These capture the start-up singularity $u_t\sim t^{\alpha-1}$ and are standard for graded-mesh L1 analysis \cite{Stynes17}.

\begin{theorem}[Convergence]\label{thm:convergence-graded}
Let $t_k=T(k/M)^r$ with $r\ge1$ and $\tau_k=t_k-t_{k-1}$. Consider the fully implicit scheme \eqref{3.3} with the L1 discretisation of the Caputo derivative on $\{t_k\}_{k=0}^{M}$.  Suppose the above regularity assumption holds. Then, the numerical solution converges to the exact solution with the error bound:
\[
\max_{0\le k\le M}\|u(\cdot,t_k)-u^k\|_\infty
\;\le\;
C\Big(M^{-\min\{\,2-\alpha,\;r\alpha\,\}} + h^2\Big),
\]
where  $C>0$  is a constant independent of $h$ and $M$.
\end{theorem}

\begin{proof}
Let \( e_i^k = u(x_i, t_k) - u_i^k \) denote the error at the grid point \( (x_i, t_k) \).
Using the graded L1 formula and subtracting the discrete scheme from the PDE at the nodes, we obtain the error equation
\begin{equation}\label{eq:error-eqn-graded}
- \frac{1}{h^2} e_{i+1}^k
+ \Big(\frac{d_{k,k}}{\Gamma(2-\alpha)} + \frac{2}{h^2} + p^k\Big) e_i^k
- \frac{1}{h^2} e_{i-1}^k
=
\frac{1}{\Gamma(2-\alpha)}  \sum_{j=1}^{k-1}\!\bigl(d_{k,j+1}-d_{k,j}\bigr)\,e^{\,k-j}_i
\;+\; \xi_i^k ,
\end{equation}
where \( \xi_i^k \) is the truncation error and the $e^0$-term vanishes because  $u_i^0 = \varphi(x_i)$ gives $e_i^0 =0$.
The truncation estimate
\begin{equation}\label{eq:trunc-graded}
\|\xi^k\|_\infty \;\le\; C_0 \Big(M^{-\min\{2-\alpha,\;r\alpha\}} + h^2\Big),\qquad 1\le k\le M.
\end{equation}
holds uniformly in $k$ by the time-weighted regularity and graded-mesh $L1$ analysis \cite{Stynes17}.

Collecting the interior components of the equation \eqref{eq:error-eqn-graded}, we obtain the matrix form
\[
   A^{(k)}e^k
   = \frac{1}{\Gamma(2-\alpha)}
     \sum_{j=1}^{k-1}\!\bigl(d_{k,j+1} - d_{k,j}\bigr)e^{\,k-j}
     +\xi^k ,
\]
where $A^{(k)}$ is the strictly diagonally dominant $M$–matrix from Theorem~\ref{thm:stability} with
$d_{k,k}=\tau_k^{-\alpha}$ and $\tau_k=t_k-t_{k-1}$.
Hence $(A^{(k)})^{-1}\ge0$ and, by a discrete maximum-principle argument,
\begin{equation}\label{eq:inv-graded}
\bigl\|(A^{(k)})^{-1}\bigr\|_\infty
\;\le\; \frac{\Gamma(2-\alpha)}{d_{k,k}}
\;=\; \Gamma(2-\alpha)\,\tau_k^{\alpha}.
\end{equation}

Taking  $\|\cdot\|_\infty$ norms and using \eqref{eq:inv-graded} gives
\[
  \|e^k\|_\infty
  \;\le\; \tau_k^{\alpha}
      \sum_{j=1}^{k-1}\!\bigl(d_{k,j+1} - d_{k,j}\bigr)\,\|e^{\,k-j}\|_\infty
      + \Gamma(2-\alpha)\,\tau_k^{\alpha}\,\|\xi^k\|_\infty .
\]
Define
\(
   E_k:=\max_{0\le m\le k}\|e^{m}\|_\infty
\)
and note that $\sum_{j=1}^{k-1}(d_{k,j+1} - d_{k,j}) = d_{k,k}-d_{k,1}$.
Then
\[
  E_k
  \;\le\;\Big(1-\Theta_k\Big)\,E_{k-1}
          +\Gamma(2-\alpha)\,\tau_k^{\alpha}\,\|\xi^k\|_\infty,
  \qquad
  \Theta_k:=\tau_k^{\alpha}\,d_{k,1}\in(0,1].
\]
Using \eqref{eq:trunc-graded} and $\tau_k^\alpha\le T^\alpha$,
\[
  E_k \;\le\; (1-\Theta_k)E_{k-1}
  + \widetilde C \Big(M^{-\min\{2-\alpha,\;r\alpha\}} + h^2\Big),
  \quad \widetilde C:=\Gamma(2-\alpha)T^\alpha C_0 .
\]
Since $0<1-\Theta_k<1$ and $\{\Theta_k\}$ is nonnegative, a variable-coefficient discrete
Grönwall yields
\[
  E_k \;\le\; C \Big(M^{-\min\{2-\alpha,\;r\alpha\}} + h^2\Big),
  \qquad 0\le k\le M,
\]
with $C$ independent of $h$ and $M$. Finally,
\(
\|u(\cdot,t_k)-u^{k}\|_\infty=\|e^{k}\|_\infty\le E_k
\),
which completes the proof.
\end{proof}

\section{Determination of the unknown coefficient $p(t)$}\label{5}
Consider the inverse problem governed by the diffusion equation \eqref{1.1}
subject to homogeneous Dirichlet boundary conditions $u(0,t)=u(l,t)=0$ for $t\in[0,T]$ and an initial condition $u(x,0)=\varphi(x)$ for $x\in[0,l]$. The function $p(t)$ depends only on $t$ and can be identified from the additional integral overdetermination condition \eqref{1.4}.
To derive a formula for $p(t)$, we first apply the Caputo fractional derivative 
$\partial^{\alpha}_{t}$ to the integral condition \eqref{1.4}.

We assume that \( f(x,t) \in C(\overline{Q}_T) \), \( \varphi(x) \in C[0,l] \), and \( g(t) \in C^1[0,T] \), with \( \partial_t^\alpha g(t) \in C[0,T] \), ensuring that all terms in the inverse problem formulation are well-defined.

\[
\partial^{\alpha}_{t} \Bigl(\int_{0}^{l} u(x,t)\omega(x)\ dx \Bigr)
= \int_{0}^{l} \partial^{\alpha}_{t} u(x,t)\,\omega(x)\\,dx = \partial^{\alpha}_{t} g(t).
\]
By substituting the PDE \eqref{1.1}, namely 
\(
\partial^{\alpha}_{t} u = u_{xx} \;-\; p(t)\,u \;+\; f(x,t),
\)
into the above integral, one obtains
\[
\partial^{\alpha}_{t} g(t) = \int_{0}^{l} u_{xx}(x,t)\,\omega(x)\,dx 
- p(t) \int_{0}^{l} u(x,t)\,\omega(x)\,dx 
+ \int_{0}^{l} f(x,t)\,\omega(x)\,dx.
\]
The integral $\int_{0}^{l} u_{xx}(x,t)\,\omega(x)\,dx $ is treated using the integrating by parts twice:
\[
\int_{0}^{l} u_{xx}(x,t)\,\omega(x)\,dx = [u_x \omega]^l_0-[u\omega']^l_0 +\int_{0}^{l} u(x,t)\,\omega''(x)\,dx.
\]
Given the boundary conditions \( u(0,t) = u(l,t) = 0 \), the terms involving \( u\omega' \) vanish. Thus, we have:
\[
\int_{0}^{l} u_{xx}(x,t)\,\omega(x)\,dx =u_x(l,t)\omega(l)-u_x(0,t)\omega(0) +\int_{0}^{l} u(x,t)\,\omega''(x)\,dx.
\]
Noting that$\int_{0}^{l} u(x,t)\,\omega(x)\,dx = g(t)$, we rearrange to obtain
\begin{equation}
\label{Pformula}
p(t)=\frac{u_x(l,t)\omega(l)-u_x(0,t)\omega(0) +\int_{0}^{l} u(x,t)\,\omega''(x)\,dx
+\int_{0}^{l} f(x,t)\,\omega(x)\,dx - \partial^{\alpha}_{t} g(t)}{g(t)}.
\end{equation}
here, we have assumed that $g(t) \neq 0$.
The discrete form of formula \eqref{Pformula} is obtained using second-order one-sided finite differences for spatial derivatives at the boundaries, and composite Simpson’s rule for the integrals. The second derivative of the weight function is computed analytically, and the Caputo fractional derivative of $g(t)$ is evaluated exactly using its known analytical expression.

Hence, once $u(x,t)$  is computed at each time level using a numerical method, the coefficient $p(t)$ can be updated by evaluating the integrals and derivatives on the right-hand side of \eqref{Pformula}.
Using these equations above, Algorithm \ref{alg1} has been developed to simultaneously compute $u(x,t)$ and the unknown coefficient $p(t)$

\begin{algorithm}[H]
\caption{Identification of $\{p(t), u(x,t)\}$ in the time-fractional diffusion equation}
\label{alg1}
\begin{algorithmic}[1]
\Require Domain length $\,l\,$, final time $\,T$
\Require Fractional order $\,\alpha\in(0,1)$, grading exponent $\,r\ge 1$ (default $r=2-\alpha$)
\Require Number of spatial nodes \( N+1 \), number of time steps \( M+1 \)
\Require Known functions: $\,\varphi(x)$, $f(x,t)$, $\omega(x)$, $g(t)$, and $\,\partial_t^\alpha g(t)$
\State Spatial mesh: $h=l/N$, $x_i=i\,h$, $i=0,\dots,N$.
\State Temporal mesh: $t_k=T\,(k/M)^r$, $k=0,\dots,M$, with $\tau_k=t_k-t_{k-1}$ for $k\ge1$.
\State Initialise: $u_i^0=\varphi(x_i)$ for all $i$.
\State Compute \( p^0 \) using the integral formula~\eqref{Pformula}
\For{$k=1$ \textbf{to} $M$}
    \State Set current time $t_k$ and steps $\{\Delta t_j\}_{j=1}^k$.
    \State Compute non-uniform L1 weights $b_{k,j}$ for $j=1,\dots,k$:
  \[
     d_{k,j} \;=\; \frac{(t_k-t_{j-1})^{1-\alpha} - (t_k-t_j)^{1-\alpha}}{\Gamma(2-\alpha)\,\Delta t_j}\,.
  \]
  \State Construct the tridiagonal matrix \( A \in \mathbb{R}^{(N-1)\times(N-1)} \) and compute the right-hand side \( b_i \) from the finite difference scheme \eqref{3.3}
   \State Solve the linear system \( A u^k = b \) (using the Thomas method), and set boundary values \( u_0^k = u_N^k = 0 \)
    \State Update \( p^k \) using the integral expression~\eqref{Pformula}.
\EndFor
\Statex \textbf{Output:} Numerical solution \( \{ u_i^k \}_{i,k=0}^{N,M} \) and coefficient \( \{ p^k \}_{k=0}^{M} \)
\end{algorithmic}
\end{algorithm}

\section{Numerical experiments}\label{6}
To validate the theoretical analysis, we conducted a series of numerical experiments in this section to assess the accuracy and convergence behavior of the proposed method. The performance was evaluated using two standard error metrics: the maximum absolute error and the 
$L^2$ -norm error.
\subsection{Numerical example}
This subsection will test the proposed method by considering the following numerical example. Consider the fractional equation \eqref{1.1} 
 with the following source term:
\[
 f(x,t) = sin(\pi x) \biggl(\Gamma(2 - \alpha)\Gamma(1 + \alpha) + \pi^2\Gamma(2 - \alpha)(1+t^{\alpha}) + 1 \biggr)
\]
We assume that the domain is \( x \in [0,1] \) and \( t \in [0,1] \) and input data:     
\[
 u(x, 0) = \Gamma(2 - \alpha)sin(\pi x), \quad \omega(x)=sin(\pi x),\quad g(t) = \frac{1}{2}\Gamma(2 - \alpha)(1+t^{\alpha}). 
\]
The exact solution of equation \eqref{1.1} in this case is 
\[
u(x,t) = \Gamma(2 - \alpha)(1+t^{\alpha})sin(\pi x), \quad p(t) = \frac{1}{\Gamma(2 - \alpha)(1+t^{\alpha})}.
\]
Figures \ref{fig1},\ref{fig2}, and \ref{fig3} illustrate the comparisons between the analytical solution and the numerical results calculated for $\{p(t), u(x,t)\}$ with various values of $\alpha$.

\begin{figure}[h]
\centering
\includegraphics[width=0.8\textwidth]{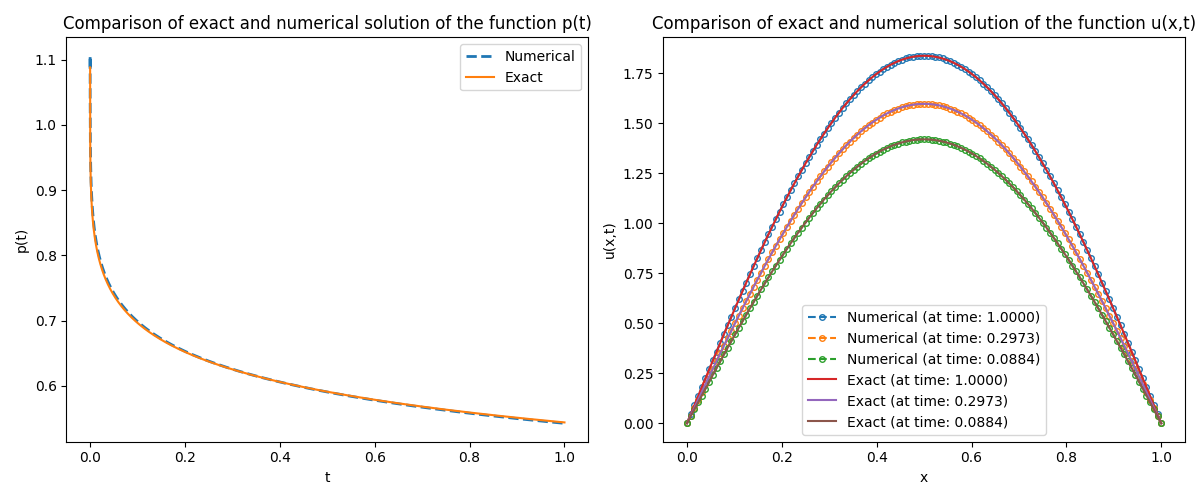}
\caption{The analytical and numerical solutions of $\{p(t), u(x,t)\}$ when $\alpha=0.25$.}
\label{fig1}
\end{figure}
\begin{figure}[h]
\centering
\includegraphics[width=0.8\textwidth]{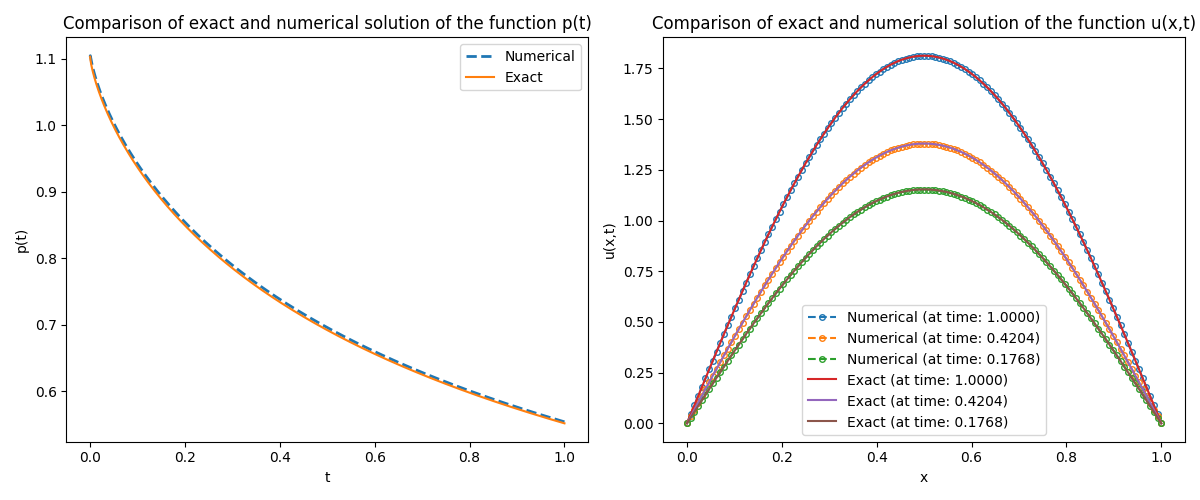}
\caption{The analytical and numerical solutions of $\{p(t), u(x,t)\}$ when  $\alpha=0.75$.}
\label{fig2}
\end{figure}
\begin{figure}[h]
\centering
\includegraphics[width=0.8\textwidth]{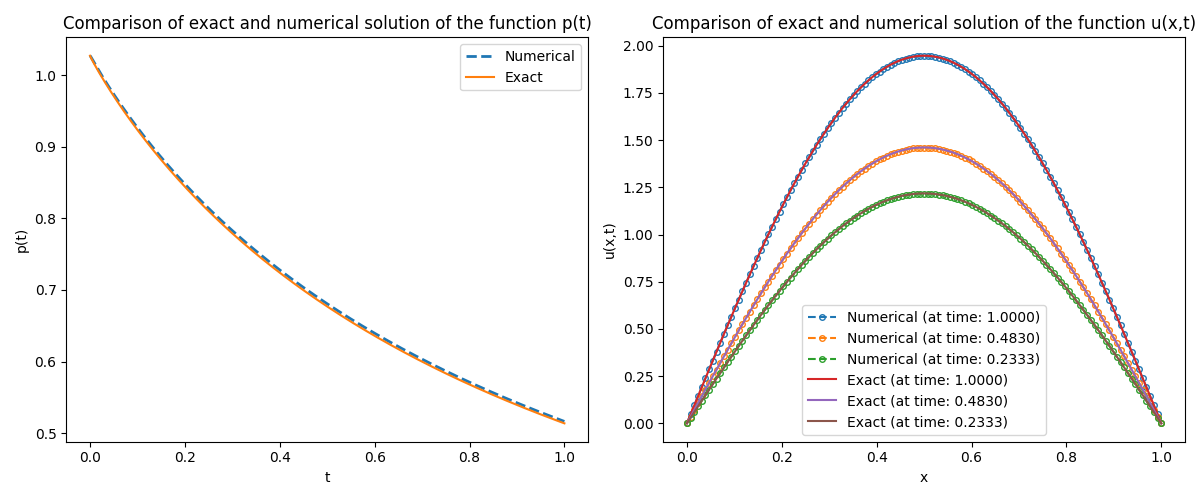}
\caption{The analytical and numerical solutions of ${p(t),u(x,t)}$ when  $\alpha=0.95$.}
\label{fig3}
\end{figure}
Figure~\ref{fig4} illustrates the numerical solutions of $\{p(t), u(x,t)\}$ corresponding to different values of the fractional order $\alpha$. 

An interesting non-monotonic behavior is observed: the amplitude of the solution decreases as $\alpha$ increases from $0.1$ to $0.5$, and then increases again from $\alpha = 0.5$ to $\alpha = 0.9$. This pattern reflects the interplay between diffusion and memory effects inherent in fractional-order models. For small values of $\alpha$, the system is dominated by strong memory, resulting in a slower diffusion process and higher retention of the initial condition, which leads to a larger amplitude. As $\alpha$ increases, the memory effect weakens and diffusion becomes more pronounced, initially reducing the amplitude. However, beyond a certain threshold, the balance between memory and diffusion shifts again, leading to an increase in the solution magnitude. This behavior highlights the complex dynamics governed by the fractional order and its impact on the physical evolution of the system.

\begin{figure}[h]
\centering
\includegraphics[width=0.95\textwidth]{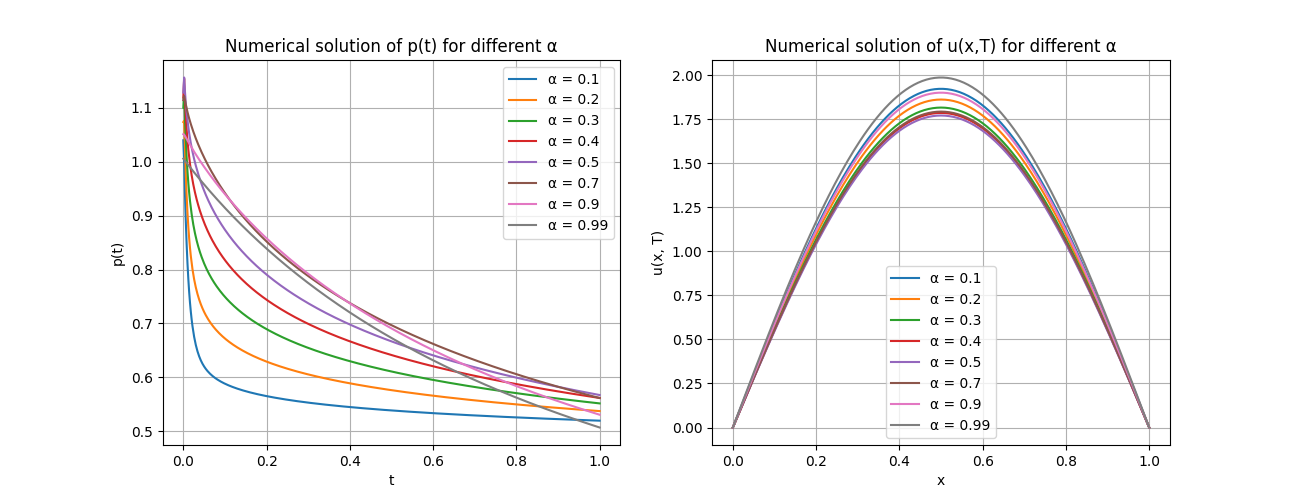}
\caption{The numerical solutions of ${p(t),u(x,t)}$ when $ T=1$ and at different $\alpha$.}
\label{fig4}
\end{figure}

Figure \ref{fig5} presents surface plots of the exact solution and the numerical approximation of $u(x,t)$ at different $\alpha$. The results are shown for a uniform grid with $N=M=200$.

\begin{figure}[h]
  \centering
  \begin{minipage}{\linewidth}
    \centering
    \includegraphics[width=0.7\linewidth]{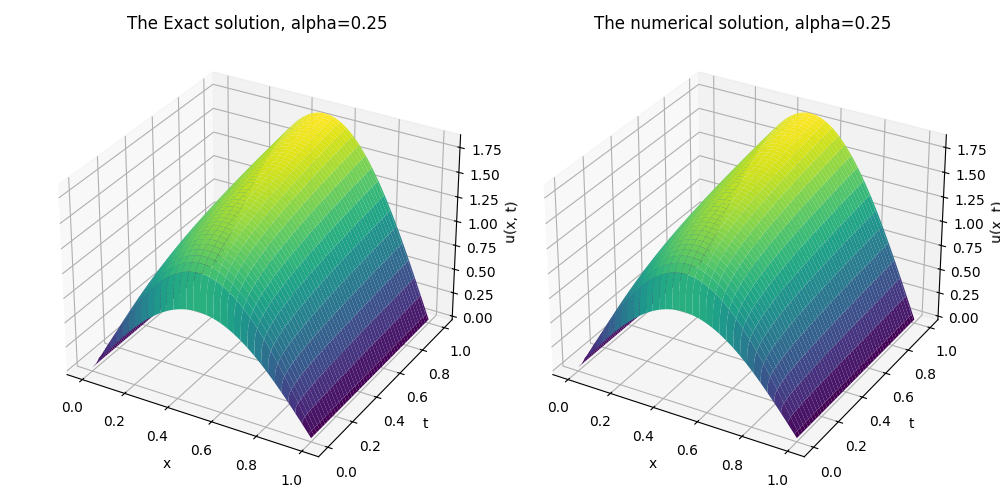}
    \label{a}
  \end{minipage}\quad
  \begin{minipage}{\linewidth}
    \centering
    \includegraphics[width=0.7\linewidth]{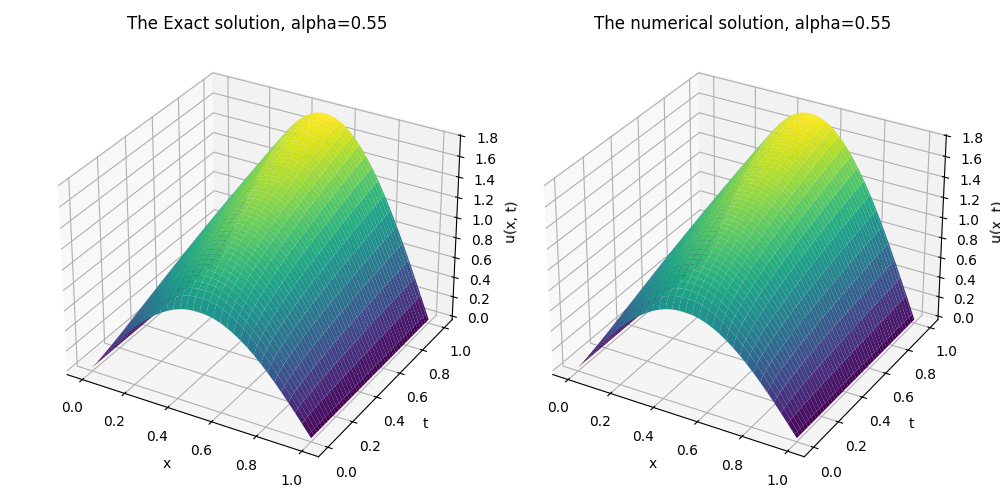}
    b
  \end{minipage}\quad
  \begin{minipage}{\linewidth}
    \centering
    \includegraphics[width=0.7\linewidth]{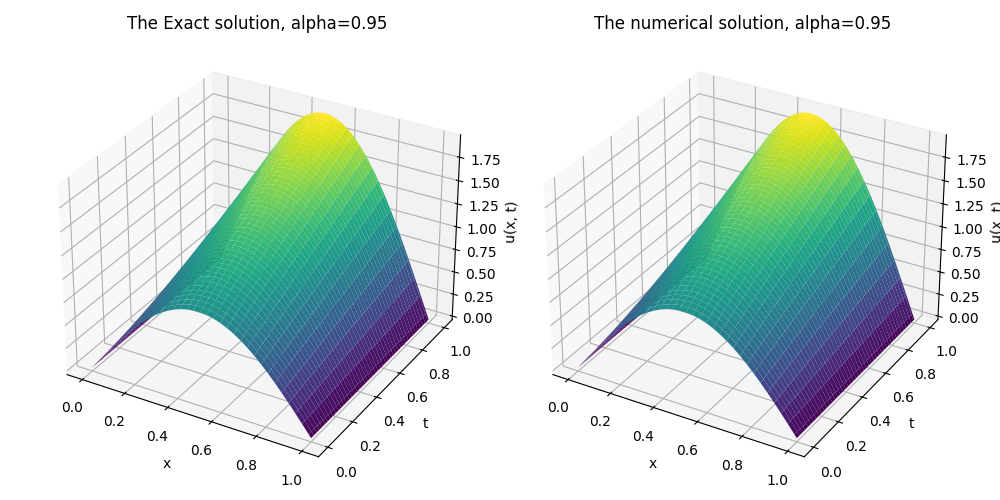}
    c
  \end{minipage}
  \caption{Comparison of the exact (left) and numerical (right) solutions for $u(x,t)$ at different values of the fractional order $\alpha$.}
  \label{fig5}
\end{figure}

Tables~\ref{table1.1} and \ref{table1.2} report the spatial and temporal convergence results obtained with Algorithm~\ref{alg1}. In Table \ref{table1.1}, spatial convergence is evaluated with a fixed number of time levels \(M=1000\).
As the number of spatial grid points \(N\) increases (i.e., \(h=1/N\) decreases), both the max-norm and \(L^2\)-norm errors for the numerical solution \(u(x,t)\) decay steadily, exhibiting second-order convergence in space. These results validate the accuracy of the numerical approach in line with theoretical expectations.

\begin{table}[htbp]
\centering
\caption{Maximum norm and \(L^2\)-norm errors versus the number of spatial grid points \(N\) (time levels fixed at \(M=1000\)).}
\begin{tabular}{c c c c c c}
\hline
$\alpha$ & $N$ & maximum error & $CO_h$ &  $L^2$ error & $  CO_h$ in $L^2$ \\
\hline
\multirow{3}{*}{0.35} 
& 16 & 4.582e-02 &         & 3.240e-02 &         \\
& 32 & 9.852e-03 & 2.218      &  6.967e-03 &  2.218  \\
& 64 & 2.304e-03 & 2.097       & 1.629e-03 & 2.097  \\
& 128 & 4.827e-04 & 2.255        & 3.413e-04 & 2.255   \\
\hline
\multirow{3}{*}{0.5} 
& 16 & 4.246e-02 &         & 3.003e-02 &         \\
& 32 & 9.192e-03 & 2.208      & 6.500e-03 &  2.208   \\
& 64 & 1.987e-03 & 2.210      & 1.405e-03 &  2.210   \\
& 128 & 2.402e-04 & 3.048       & 1.698e-04 & 3.048   \\
\hline
\multirow{3}{*}{0.75} 
& 16 & 3.925e-02 &         &  2.776e-02 &         \\
& 32 &  8.677e-03 &  2.178     & 6.135e-03 &  2.178   \\
& 64 & 1.759e-03 & 2.302  & 1.244e-03 & 2.302   \\
& 128 & 6.914e-05 & 4.669     & 4.889e-05 & 4.669   \\
\hline
\end{tabular} \label{table1.1}
\end{table}

In Table~\ref{table1.2}, temporal convergence is assessed with a fixed spatial grid \(N=128\) on a graded time mesh \(t_k=T(k/M)^r\). As the number of time levels \(M\) increases (so the maximum local time step decreases), the errors decrease monotonically, with empirical rates close to \(O\!\big((\max_k \Delta t_k)^{\,2-\alpha}\big)\). These results confirm that the graded-mesh discretisation achieves the expected temporal accuracy for the chosen test problem.

\begin{table}[htbp]
\centering
\caption{Maximum norm and \(L^2\)-norm errors versus the number of time levels \(M\) (spatial grid fixed at \(N=128\)).}
\begin{tabular}{c c c c c c}
\hline
$\alpha$ & $M$ & maximum error & $CO_\tau$ &  $L^2$ error & $  CO_\tau$ in $L^2$ \\
\hline
\multirow{3}{*}{0.35} 
& 64 & 3.772e-03 &        &  2.667e-03 &    \\
& 128 & 1.458e-03 & 1.372  & 1.031e-03 & 1.372    \\
& 256 & 3.939e-04 & 1.888      & 2.786e-04 &  1.888   \\
& 512 & 1.136e-04 & 1.794     & 8.033e-05 & 1.794   \\

\hline
\multirow{3}{*}{0.5} 
& 64 & 5.292e-03 &         &  3.742e-03 &    \\
& 128 & 2.242e-03 & 1.239      & 1.585e-03 &   1.239   \\
& 256 & 7.930e-04 &   1.499    & 5.608e-04  &   1.499 \\
& 512 & 9.342e-05 &    3.086  & 6.606e-05 &   3.086 \\
\hline
\multirow{3}{*}{0.75} 
& 64 & 7.436e-03 &         & 5.258e-03 &    \\
& 128 & 3.432e-03 &    1.116     & 2.427e-03 &   1.116    \\
& 256 & 1.416e-03 & 1.277     &  1.001e-03 & 1.277   \\
& 512 & 4.145e-04 & 1.773    & 2.931e-04 & 1.773  \\

\hline
\end{tabular} \label{table1.2}
\end{table}
The accuracy of the reconstructed coefficient $p(t)$, as given by formula~\eqref{Pformula}, depends on multiple factors: the precision of the numerical solution $u(x,t)$, the finite difference approximations of the spatial derivatives $u_x(0,t)$ and $u_x(l,t)$, and the numerical evaluation of the integral terms. Inaccuracies in any of these components can influence the quality of the reconstructed coefficient.

\begin{table}[H]
\centering
\caption{Maximum and $L^2$ errors in the identified coefficient $p(t)$ for different values of $\alpha$ and grid resolutions $N = M$.}
\begin{tabular}{cccc}
\hline
$\alpha$ & $N = M$ & maximum error in $p$ & $L^2$ error in $p$ \\
\hline
0.25 & 128 & 1.369e-01 & 1.393e-02 \\
     & 256 & 1.082e-01 & 6.763e-03 \\
     & 512 & 8.472e-02 & 3.270e-03 \\
\hline
0.50 & 128 & 4.837e-02 & 1.844e-02 \\
     & 256 & 3.032e-02 & 9.223e-03 \\
     & 512 & 1.865e-02 & 4.545e-03 \\
\hline
0.75 & 128 & 2.394e-02 & 2.187e-02 \\
     & 256 & 1.212e-02 & 1.130e-02\\
     & 512 & 6.003e-03 & 5.673e-03 \\
\hline
0.95 & 128 & 1.943e-02 & 1.757e-02 \\
     & 256 & 1.054e-02 & 9.552e-03 \\
     & 512 & 5.504e-03 & 5.006e-03 \\
\hline
\end{tabular}
\label{table1.3}
\end{table}

As shown in Table~\ref{table1.3}, the maximum and $L^2$ errors in the identified coefficient $p(t)$ decrease with mesh refinement for all values of $\alpha$. The accuracy improves as $\alpha$ increases, reflecting the enhanced stability and convergence of the method when approaching the classical diffusion limit. \noindent
The relatively lower accuracy in the reconstruction of $p(t)$ is primarily attributed to the singular behavior of the solution near $t = 0$, which is characteristic of time-fractional diffusion equations. This singularity affects the numerical evaluation of both the fractional derivative and the integral terms, especially at early times. However, as observed in Table~\ref{table1.3}, the errors decrease with increasing $N$ and $M$, indicating that higher spatial and temporal resolutions can significantly improve the accuracy of $p(t)$.

Overall, the results demonstrate strong agreement between the numerical and exact solutions, as reflected by the small maximum and $L^2$ errors for both $u(x,t)$ and $p(t)$. These findings confirm the reliability and accuracy of the proposed method for solving the fractional diffusion equation with a time-dependent coefficient. In the following subsection, we further evaluate the robustness of the method under noisy data conditions.

\subsection{Effect of noisy data on the reconstruction of $p(t)$}

To evaluate the robustness of the proposed numerical method under practical conditions, we conducted a series of experiments in which the integral overdetermination data \( g(t) \) and its time-fractional derivative were contaminated with additive noise. Specifically, we considered the noisy data
\[
g^{\delta}(t) = g(t) + \delta \cdot \eta(t), \quad
\partial_t^{\alpha}g^{\delta}(t) = \partial_t^{\alpha}g(t) + \delta \cdot \xi(t),
\]
where \( \delta \in \{0.01,\, 0.03,\, 0.05\} \) represents the relative noise level (1\%, 3\%, and 5\%, respectively), and \( \eta(t),\, \xi(t) \) are random perturbation functions generated to simulate measurement errors.

The noisy data \( g^{\delta}(t) \) and \( \partial_t^{\alpha}g^{\delta}(t) \) were used in place of the exact values in the coefficient recovery formula. The results are presented in Figures~\ref{fig:noise1}, \ref{fig:noise3}, and \ref{fig:noise5}, where each plot compares the exact and numerically reconstructed values of \( p(t) \) under different noise levels.

\begin{figure}[h]
    \centering
    \includegraphics[width=1.2 \textwidth]{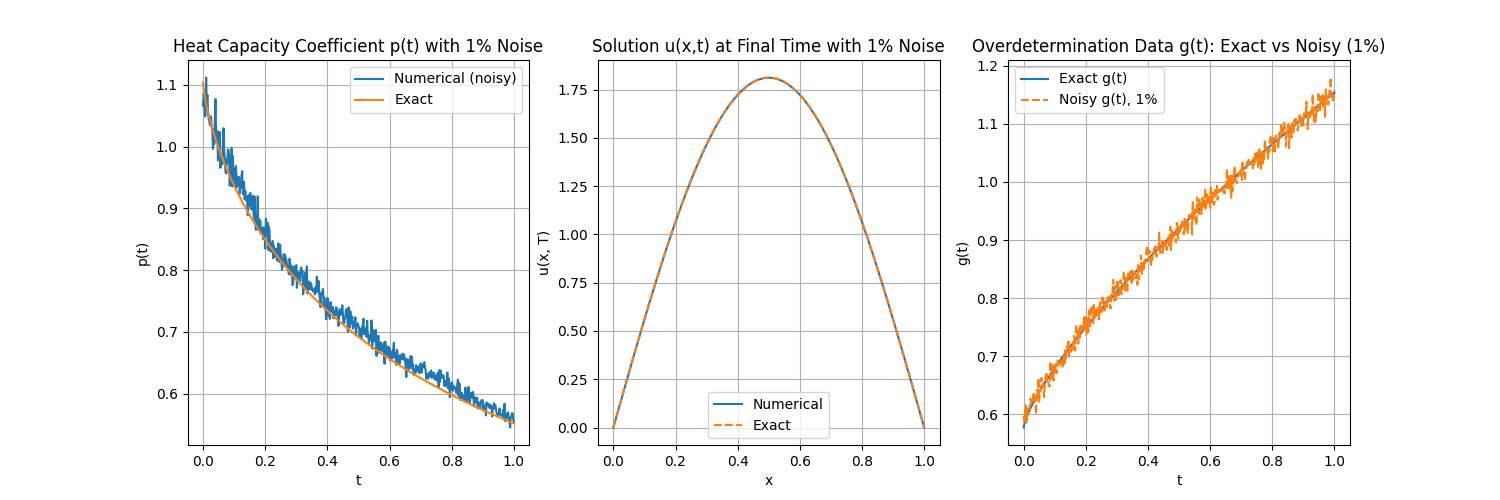}
    \caption{Solutions of \( p(t), u(x,t) \) with 1\% noisy overdetermination data.}
    \label{fig:noise1}
\end{figure}

\begin{figure}[h]
    \centering
    \includegraphics[width=1.2 \textwidth]{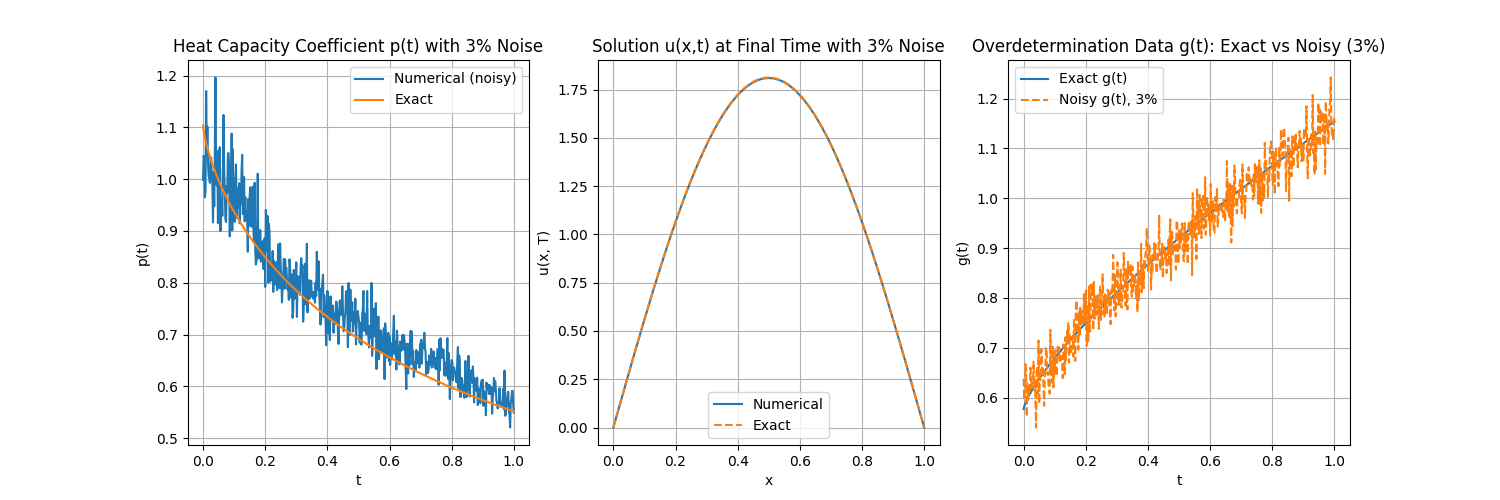}
    \caption{Solutions of \( p(t), u(x,t) \) with 3\% noisy overdetermination data.}
    \label{fig:noise3}
\end{figure}

\begin{figure}[h]
    \centering
    \includegraphics[width=1.2 \textwidth]{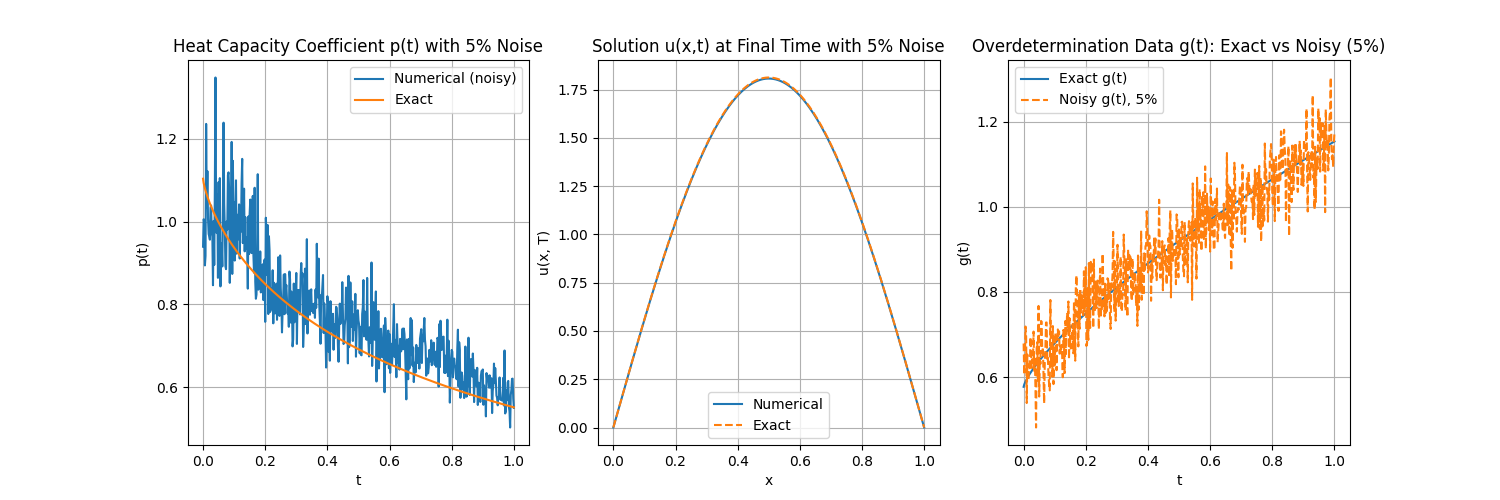}
    \caption{Solutions of \( p(t), u(x,t) \) with 5\% noisy overdetermination data.}
    \label{fig:noise5}
\end{figure}

The results show that the proposed method exhibits good stability with respect to data noise. For a low noise level of 1\%, the reconstructed coefficient \( p(t) \) closely follows the exact solution throughout the interval \( [0,1] \). As the noise level increases to 3\% and 5\%, the error in the reconstruction becomes more noticeable, especially near \( t=0 \), where the sensitivity of the fractional derivative is more pronounced. Nevertheless, even under 5\% noise, the overall trend of \( p(t) \) is captured accurately, and the method remains stable.

These findings confirm the method's capability to handle noisy data effectively, making it suitable for practical applications where exact measurements are rarely available.

\section{Conclusion}\label{7}
We proposed a robust numerical framework to recover the time-dependent coefficient in a time-fractional diffusion equation. Our method is supported by theoretical stability and convergence analysis and is validated by accurate numerical simulations. The results demonstrate strong performance even under noisy conditions, highlighting the potential of the method for practical applications.

\section*{Funding}
\noindent
The research is financially supported by a grant from the 
Ministry of Science and Higher Education of the Republic of Kazakhstan (No. AP27508473), by the FWO Research Grant G083525N: Evolutionary partial differential equations with strong singularities, and by the Methusalem programme of the Ghent University Special Research Fund (BOF) (Grant number 01M01021).


\end{document}